\documentclass[a4paper,12pt]{article}

\usepackage[latin1]{inputenc}
\usepackage{latexsym,amsfonts,amssymb,amsmath}

\newtheorem{thm}{Theorem}[section] 

\newtheorem{prop}[thm]{Proposition} 
\newenvironment{proof}{\noindent{\bf Proof.}}{\noindent$\Box$\par\medskip} 
\newtheorem{rmk}[thm]{Remark} 
\newtheorem{dfn}[thm]{Definition}

\newcommand\sudda[1]{}
\newcommand\bpil{\buildrel b\over\longrightarrow} 
\newcommand\bprimpil{\buildrel
b'\over\longrightarrow} 
\newcommand\nollpil{\buildrel
0\over\longrightarrow}

\DeclareMathOperator{\Log}{Log} 
\DeclareMathOperator{\Lie}{Lie}
\DeclareMathOperator{\Exp}{Exp}
 
\DeclareMathOperator{\ord}{ord}

\newcommand\ideal[1]{\langle#1\rangle}

\DeclareMathOperator{\ho}{H}
\DeclareMathOperator{\hho}{HH}
\DeclareMathOperator{\hc}{HC}
\DeclareMathOperator{\hcc}{hc}
\DeclareMathOperator{\bb}{B}
\DeclareMathOperator{\cc}{C}

\DeclareMathOperator{\hcfree}{hcfree}
\DeclareMathOperator{\rhc}{\overline{HC}}

\DeclareMathOperator{\coker}{coker}
\DeclareMathOperator{\tor}{Tor}
\DeclareMathOperator{\ext}{Ext}

\DeclareMathOperator{\im}{im}

\DeclareMathOperator{\ch}{char}
\DeclareMathOperator{\len}{length}

\title{{\Huge Cyclic homology of algebras of global dimension at most two}}
\author{ 
\\Seminar on Cyclic Homology, November 7, 2017, held at\\
\\STOCKHOLM UNIVERSITY\\
\\by\\
\\Clas L\"ofwall}

\date{}

\begin{document}
\maketitle
\begin{abstract} We study graded connected algebras over a field of characteristic zero and give an explicit formula for the cyclic homology of a tensor algebra. By means of a slightly new definition of David Anick's notion "strongly free" we are able to prove that cyclic homology of an algebra of global dimension two is zero in homological degree greater than one and is zero also in homological degree equal to one in case the relations are monomials. We give also explicit formulas for the cyclic homology of a tensor algebra modulo one symmetric quadratic form.   
\end{abstract}

\section{Prerequisites}
\subsection{Definition of homology}
We will study $\mathbb N\times \mathbb Z_2$-graded connected algebras, $A=k\oplus I$ where $k$ is a field of characteristic zero, $I=\oplus_{q\ge1,\epsilon\in\mathbb Z_2}I_{q,\epsilon}$,\ $|I_{q,\epsilon}|<\infty$,
\ $I_{q,\epsilon}\cdot I_{q',\epsilon'}\subseteq I_{q+q',\epsilon+\epsilon'}$.
 We use $|a|$ as notation for the $\mathbb Z_2$-degree which is called the {\em sign degree}, the first degree is called the {\em weight}. The series for $I$ is defined as
$$
I(z,y)=\sum_{q\ge1,\epsilon\in\mathbb Z_2}|I_{q,\epsilon}|\in\mathbb Z[[z]][y]/(y^2-1)
$$
The ring $\mathbb Z[[z]][y]/(y^2-1)$ is isomorphic to the ring of $\mathbb Z^+\times \mathbb Z_2$-graded locally finite dimensional vector spaces with $\oplus$ and $\otimes$ as operations and with formal additive inverses adjoined.

We use the same notation for a space and its series.

\vspace{6pt}
\noindent\textbf{Example} $T(V)\cdot V+1=T(V)$ implies that $T(V)=1/(1-V)$

\vspace{6pt}
We have $\hc_0(I)=\rhc_0(A)=I/[I,I]$ where $[I,I]$ is the subspace of $I$ (not ideal!) generated by all commutators $ab-(-1)^{|a||b|}ba$. Taking modulo $[I,I]$ is the same as allowing circular permutation of monomials, e.g., $abcd=-dabc=cdab=-bcda$ in $I/[I,I]$ if $a,b,c,d$ are odd elements in $I$.

The Hochschild homology $\hho(A)$ is obtained by tensoring the two-sided bar resolution $(A\otimes I^{\otimes n}\otimes A,b')$ of $A$ as an $A$-bimodule with the bimodule $A$. Thus $\hho(A)$ may be seen as the homology of the complex ($n\ge0$, $I^{\otimes 0}=k$) 
$$
(I^{\otimes n}\oplus I^{\otimes{n+1}},\begin{pmatrix}-b'&0\\1-t&b\end{pmatrix})
$$
where
\begin{align*}
b'(a_0,\ldots,a_{n-1})&=\sum_{i=0}^{n-2}(-1)^{\epsilon(i)}(a_0,\ldots,a_ia_{i+1},\ldots,a_{n-1})\\
t(a_1,\ldots,a_n)&=(-1)^\eta(a_n,a_1,\ldots,a_{n-1})\\
b(a_0,\ldots,a_n)&=\sum_{i=0}^{n-1}(-1)^{\epsilon(i)}(a_0,\ldots,a_ia_{i+1},\ldots,a_n)+(-1)^\epsilon(a_na_0,a_1,\ldots,a_{n-1})
\end{align*}
and where
\begin{align*}
\epsilon(i)&=i+\sum_{r=0}^i|a_r|\\
\eta&=(|a_n|+1)\sum_{r=1}^{n-1}(|a_r|+1)\\
\epsilon&=|a_n|+(|a_n|+1)\sum_{r=0}^{n-1}(|a_r|+1)
\end{align*}
This is the mapping cone  in positive degrees of the map of complexes $(1-t): (I^{\otimes n},b')\to (I^{\otimes n},b)$ where $n\ge1$, or the total complex of the corresponding double complex,
$$
\begin{matrix}\ldots&\longrightarrow
\ \ I^{\otimes 4}&\bprimpil\ \ I^{\otimes 3}&
\bprimpil\ \
I^{\otimes2}&\bprimpil\ \ I&\nollpil k \\
\\
 &{\scriptstyle(1-t)}\
\downarrow &{\scriptstyle(1-t)}\
\downarrow&{\scriptstyle(1-t)}\ \downarrow&
\ \ \ \ \ {\scriptstyle0}\,\downarrow \\
\\
\ldots&\longrightarrow\ \ I^{\otimes
4}&\bpil\ \ 
I^{\otimes 3}&\bpil\ \ I^{\otimes2}&
\bpil\ \ I
\end{matrix}
$$

Taking homology first vertically and then horizontally   
 gives a long exact sequence by spectral sequence theory. There is only one differential $d_2$. 
\begin{align*}
\ldots\to&\ho_{n+2}(\coker(1-t))\xrightarrow{\ d_2\ }\ho_{n}(\ker(1-t))\to\hho_{n}(A)\to\\
&\ho_{n+1}(\coker(1-t))\xrightarrow{\ d_2\ }\ho_{n-1}(\ker(1-t))\to\\
& \hho_{n-1}(A)\to\ho_{n}(\coker(1-t))\to\ldots
\end{align*} 
The homology of the complex $\cc_t(A)=(\coker(1-t),\bar b)$ is by definition $\rhc(A)$, with a degree shift, so $\rhc_n(A)=\ho_{n+1}((\coker(1-t),\bar b))$. The norm map defines an isomorphism $\sum t^i:(\coker(1-t),\bar b)\to (\ker(1-t),b')$ and hence $\rhc(A)$ may also be seen as the homology of $(\ker(1-t),b')$.

In our case the differential $d_2=0$, so the long exact sequence is short:
$$
0\longrightarrow\rhc_{n-1}(A)\longrightarrow\hho_{n}(A)\longrightarrow\rhc_{n}(A)\longrightarrow0\quad\text{for}\quad n>0
$$
For $n=0$ we have $\hho_0(A)=k\oplus\rhc_0(A)=k\oplus I/[I,I]$.

The series for $\hho(A)$ and $\rhc(A)$ have three variables $x$ for the homological degree, $z$ for the weight and $y$ for the sign degree. The series belong to the ring $\mathbb Z[[x,z]][y]/(y^2-1)$. The map $\rhc_{n-1}(A)\,\to \,\hho_n(A)$ has both homological degree and sign degree equal to one, hence we have the following relation
\begin{align}\label{cyc}\hho(A)(x,z,y)=1+(1+xy)\rhc(A)(x,z,y)\end{align}
where
\begin{align*}
\hho(A)(x,z,y)&=\sum_{q\ge n\ge0,\epsilon=0,1}|\hho_{n,q,\epsilon}(A)|x^nz^qy^\epsilon\\
\rhc(A)(x,z,y)&=\sum_{q-1\ge n\ge0,\epsilon=0,1}|\rhc_{n,q,\epsilon}(A)|x^nz^qy^\epsilon.
\end{align*}

For commutative algebras $A,B$ we have
\begin{align}\label{tensor}
\hho(A\otimes B)=\hho(A)\otimes \hho(B)
\end{align}

\subsection{Koszul duality}
There is a subcomplex of the two-sided bar resolution of the form $A\otimes (A^!)^*\otimes A$, where $A^!$ is the Koszul dual of $A$. The differential in this subcomplex has only two terms with no middle terms. In fact, $(A^!)^*$ in homological degree $n$ is exactly the subspace of $I_1^{\otimes n}$ consisting of the elements for which all middle terms of the differential are zero. Thus
$$
(A^!)^*=\bigcap_{i=0}^{n-2} I_1^{\otimes\,i}\otimes\ker (\mu)\otimes I_1^{\otimes (n-2-i)}
$$
and where $\mu$ is multiplication $I_1\otimes I_1\to I_2$.

If $A$ is Koszul, then this complex is exact (the opposite is also true) and we may use this to compute the Hochschild homology as the homology of $A\otimes (A^!)^*$ where the differential again has two terms but now one of them is circulating. Taking the vector space dual of the complex we get the complex $A^*\otimes A^!$ with isomorphic homology and since $(A^!)^!=A$, this is the complex we get if we compute the Hochschild homology of 
$A^!$ with the same method. 

Hence we get that $\hho(A)$ and $\hho(A^!)$ are in a sense dual. One has to be careful with the gradings, the homological degree of 
$A_r\otimes ((A^!)^*)_p$ is $p$ and the weight is $p+r$, while in 
$(A^*)_r\otimes (A^!)_p$ the homological degree is $r$ and the weight is $p+r$. The sign degree is preserved in passing from one complex to the other. We thus get the following relation
\begin{align}\label{hhkos}
\hho_{p,q,\epsilon}(A^!)\cong\hho_{q-p,q,\epsilon}(A)
\end{align}
yielding the following formula of Hilbert series
\begin{align}\label{kos}
\hho(A^!)(x,z,y)=\hho(A)(1/x,xz,y).
\end{align}
We may obtain a similar formula for $\rhc$ using (\ref{cyc}),
\begin{align}\label{hckos}
\rhc(A^!)(x,z,y)=\frac{y}{x}\rhc(A)(1/x,xz,y),
\end{align}
which gives the following 
\begin{align}\label{kosrel}
\rhc_{p,q,\epsilon}(A^!)\cong\rhc_{q-p-1,q,\epsilon+1}(A).
\end{align}
As an application, we get an easy proof of the Hochschild-Kostant-Rosenberg theorem in the graded case. Let $A=k[x_1,\ldots,x_n]$ be a polynomial ring in $n$ even variables of weight one. Then it is easy to see that the differential in the two-sided Koszul complex $A\otimes (A^!)^*$ is zero giving that, as a vector space,
$$
\hho(A)=k[x_1,\ldots,x_n]\otimes k[\bar x_1,\ldots,\bar x_n]
$$
where $\bar x_i$ are odd anti-commutative variables of homological degree and weight one. This gives the following series in three variables,
\begin{align}\label{hkr}
\hho(A)(x,z,y)=(1+yxz)^n/(1-z)^n.
\end{align}
Using (\ref{kos}) we get
$$
\hho(A^!)(x,z,y)=(1+yz)^n/(1-xz)^n
$$
which (as expected) is the series for the exterior algebra on $n$ odd variables of weight one tensor with the polynomial algebra on $n$ even variables of homological degree and weight one.

Putting $p=0$ in (\ref{hhkos}) we get $\oplus_q\hho_{qq}(A)\cong A^!/[A^!,A^!]$. This is in fact true in general and may be compared with the result that
$\oplus_q\ext^{qq}(A)=A^!$.
\begin{prop} Let $A$ be a $k$-algebra. Then 
$$
\oplus_q\hho_{qq}(A)\cong (A^!/[A^!,A^!])^*
$$
and if $A$ is commutative, the isomorphism is an isomorphism as algebras, where the right hand side is an algebra since $[A^!,A^!]$ is a co-ideal in the Hopf algebra $A^!$. 
\end{prop}\endproof

\subsection{Free models}
A free model of a $k$-algebra $A$ is a differential graded algebra $(T(V),d)$ together with a surjective map $(T(V),d)\to A$ which induces an isomorphism in homology. Here $V$ has homological degree $\ge0$ and weight $\ge1$ and a sign degree, $d$ is a derivation of homological degree $-1$ which preserves weight and changes the sign degree, $A$ has differential zero and is concentrated in homological degree zero. 
There is always a free model for $A$ (even a minimal one, but that is not important for us). 

\vspace{6pt}
\noindent\textbf{Example} $A=k\ideal{x_1,x_2}/(x_1x_2)$. Then a free model is (as follows from Proposition \ref{freeext})
$$
k\ideal{x_1,x_2,S;\ dx_1=dx_2=0,\ dS=x_1x_2}
$$
The map to $A$ maps $S$ to 0.

\vspace{6pt}

There is also a notion of relative model for a map $f:A\to B$. This is a differential graded algebra which is a free extension of $A$, $(A*T(V),d)$, together with a surjective map $(A*T(V),d)\to B$ inducing an isomorphism in homology and which equals $f$ when it is restricted to $A$. We will use the notation $A\ideal{V}$ for $A*T(V)$.

The relative cyclic homology, $\rhc(A\to B)$, of a map $f:A\to B$ is defined as the homology of the mapping cone for the map of complexes 
$\cc_t(f):\cc_t(A)\to \cc_t(B)$. The exact sequence of complexes gives rise to a long exact sequence (where $s$ denotes shift of sign degree),
\begin{align}\label{longex}
&\ldots\to\rhc_n(A)\to\rhc_n(B)\to\rhc_{n}(A\to B)\to\notag\\
&s\rhc_{n-1}(A)\to s\rhc_{n-1}(B)\to s\rhc_{n-1}(A\to B)\to\rhc_{n-2}(A)\to\ldots
\end{align} 

Free models may be used to compute cyclic homology (and also $\tor^A(k,k)$ and $\ext_A(k,k)$). Suppose $(T(V),d)$ is a model for $A$ and that $(A\ideal{V},d)$ is a relative model for a map $A\to B$, then we have the following result, which may be proven by a spectral sequence argument, using the fact that $\rhc_n(T(V))=0$ for $n>0$.
\begin{prop}\label{model}
\begin{align*}
\rhc(A)&=\ho(T(V)/([T(V),T(V)]+k),\bar d)\\
\rhc(A\to B)&=\ho(A\ideal{V}/([A\ideal{V},A\ideal{V}]+A),\bar d)
\end{align*}
where $\bar d$ is the induced differential ($d$ maps a commutator to a linear combination of commutators).
\end{prop} 

Using free models we may get another proof of the duality between $\rhc(A)$ and  
$\rhc(A^!)$ when $A$ is Koszul. Indeed, we have that the cobar construction for $A$, $(T(I^*),d)$, where $d$ is the derivation which extends the dual of the multiplication map $I\otimes I\to I$, is a (minimal) model for $A^!$ when $A$ is Koszul. From this the relation 
(\ref{kosrel}) may be deduced.

The following result by Feigin and Tsygan, see \cite{feigin}, is fundamental for us.
\begin{prop}\label{fei} Let  $A$ och  $B$ be
graded $k$-algebras. Then
\begin{align*}\rhc_n(A*B)=&\rhc_n(A)\oplus\rhc_n(B)\quad\hbox{for}\quad
n>0\\
\rhc_0(A*B)=&\rhc_0(A)\oplus\rhc_0(B)\oplus\rhc_0(T(A^+\otimes B^+))
\end{align*} 
\end{prop}
\begin{proof} Let  $R_A$ and  $R_B$ be models for 
$A$ and  $B$. Then  $R_A*R_B$ is a 
model for  $A*B$. This follows from  K\"unneth's formula and since
$R_A*R_B$ is  free as an algebra. We have
$$
(R_A*R_B)^+=R_A^+\ \oplus\ R_B^+\ \oplus\ (R_A^+\otimes R_B^+)
\ \oplus\ (R_B^+\otimes R_A^+)
\ \oplus\ (R_A^+\otimes R_B^+\otimes R_A^+)\ \oplus\ \ldots
$$     
Modulo commutators we get
$$
(R_A*R_B)^+/[,]=R_A^+/[,]\ \oplus\ R_B^+/[,]\ \oplus\ 
(R_A^+\otimes R_B^+)
\ \oplus\ ((R_A^+\otimes R_B^+)\otimes (R_A^+\otimes R_B^+))_{t_2}
\ \oplus\ \ldots
$$
where  $t_2$ indicates that the two blocks
$(R_A^+\otimes R_B^+)$ may change place. The claim now follows taking homology
and using K\"unneth's formula and that the  functor 
$\coker(1-t):(\cdot)^{\otimes n}\to(\cdot)^{\otimes n})$ is exact, since it is right exact and isomorphic to the left exact functor $\ker(1-t):(\cdot)^{\otimes n}\to(\cdot)^{\otimes n})$. 

\end{proof}
\section{Series and Logarithms}
In this section we will get a formula for the series of
$$
\rhc_0(T(V))=T(V)^+/[T(V),T(V)]
$$ 
for a $\mathbb Z^+\times \mathbb Z_2$-graded locally finite dimensional vector space $V$. On the way we also get a formula for the series of a Lie algebra given the series for its enveloping algebra.

Let $P$ denote the additive group of formal power series 
$$
\sum_{n\ge1}a_nz^n+y\sum_{n\ge1}b_nz^n\in \mathbb Z[[z]][y]/(y^2-1)
$$
and let $U$ denote the multiplicative group $1+P.$ The inverse of
$1-\sum_{n\ge1}a_nz^n-y\sum_{n\ge1}b_nz^n$ is
$$
1+(\sum_{n\ge1}a_nz^n+y\sum_{n\ge1}b_nz^n)+(\sum_{n\ge1}a_nz^n+y\sum_{n\ge1}b_nz^n)^2+\ldots
$$
(observe that any series in $P$ may be inserted in any formal power
series in one variable to get a new well-defined series).

A topology on $P$ and $U$ is defined as follows. We say that
$f_n\to f$ if, for any $r,$ the coefficients in $f_n$ and $f$ are the
same up to $z^r$ and $yz^r$ for big enough $n$. 

Instead of assuming that the coefficients in the formal power series
are integers, we will also consider the case when the coefficients are
rational numbers. In this case we write $P^{\mathbb Q}$ and $U^{\mathbb
Q}$ for the corresponding additive and multiplicative groups.

\begin{dfn} 
A continuous group isomorphism $P\to U$ (or
$P^{\mathbb Q}\to U^{\mathbb Q}$) is
called an ``exponential'' and a continuous isomorphism $U\to P$ (or
$U^{\mathbb Q}\to P^{\mathbb Q}$) is
called a ``logarithm''.
\end{dfn}

\vspace{6pt}
\begin{prop}
\begin{align*}
&\exp(X)=\sum_{k\ge0}\frac{X^k}{k!},\quad X\in P^{\mathbb Q},
\quad\text{is an exponential }P^{\mathbb Q}\to U^{\mathbb Q}\\ 
&\log(X)=-\sum_{k\ge1}\frac{(1-X)^k}{k},\quad X\in U^{\mathbb Q},
\quad\text{is a logarithm }U^{\mathbb Q}\to P^{\mathbb Q}\\
&S(X)=\prod_{k=1}^\infty\frac{(1+yz^k)^{b_k}}{(1-z^k)^{a_k}},\
X=\sum_{k=1}^\infty a_kz^k+y\sum_{k=1}^\infty b_kz^k
\ \text{ is an exponential }P\to U
\end{align*} 
\end{prop}
\begin{proof}

The usual proof of $e^{x+y}=e^x\cdot e^y$ using formal power series depends
only on the binomial theorem and that $x$ and $y$ are commuting
symbols. Hence also $\exp(X+Y)=\exp(X)\cdot\exp(Y),$ since $P^{\mathbb Q}$
is commutative. The identities $\exp(\log(X))=X$ and $\log(\exp(X))=X$
hold, since they hold for small enough convergent series with complex coefficients.
Hence $\exp$ and $\log$ are bijective and $\log$ is also a
homomorphism and clearly they are continuous. The series $S(X)$ is the series for the symmetric algebra of a graded and
weighted vector space $X=\oplus_{i\ge1}X_i$. 
(By means of the Taylor expansion of $(1+z)^{\alpha}$ the operator $S$
may be extended to a transformation $P^{\mathbb Q}\to U^{\mathbb Q}.$)

This gives indeed an exponential 
$P\to U$.
It is evident that $S$ is a homomor\-phism and that it is continuous. To
prove bijectivity, suppose $f\in U$ and suppose inductively that there are uniquely
determined \break $a_1,b_1,\ldots,a_{k-1},b_{k-1}\in \mathbb Z$ such that 
$$
\prod_{j=1}^{k-1}\frac{(1+yz^j)^{b_j}}{(1-z^j)^{a_j}}\equiv f(z,y)
\mod z^k
$$
Hence 
$$
\prod_{j=1}^{k-1}\frac{(1-z^j)^{a_j}}{(1+yz^j)^{b_j}}f(z,y)=1+a_kz^k+b_kyz^k+\text{
higher terms}
$$
and hence 
$$
\frac{(1-z^k)^a}{(1+yz^k)^b}\prod_{j=1}^{k-1}\frac{(1-z^j)^{a_j}}{(1+yz^j)^{b_j}}f(z,y)\equiv1
\mod z^{k+1}\Leftrightarrow a=a_k\text{ and }b=b_k
$$
\end{proof}
\begin{dfn}
The inverse of $S$ is a logarithm, $\Lie(X)$, which gives back
the series for the Lie algebra when $X$ is the series of its enveloping algebra.
\end{dfn}

In particular $\Lie(\frac{1}{1-X})$ gives the series for the free
Lie algebra on a graded and weighted vector space with series $X.$

\vspace{6pt}
By means of ``$\log$'' we get a lot of other logarithms:

\begin{prop}

Let $c=\{c_k\}_{k=1}^\infty$ be any sequence of  rational numbers with
$c_1\ne0.$ Then
\begin{align}
&\Log_c(X)=\sum_{k=1}^\infty c_k\log(X(z^k,(-1)^{k+1}y^k))\text{ is a logarithm } U^{\mathbb Q}\to P^{\mathbb Q}\\
&\Log_c(1+yz)=\Log_c(1-z)\big\vert_{z=-yz}\label{yz}
\end{align}

\end{prop}
\begin{proof}
The proof of (\ref{yz}) is direct:
\begin{align*}
\Log_c(1+yz)&=\sum_{k=1}^\infty
c_k\log(1+(-1)^{k+1}y^kz^k)\\
&=\sum_{k=1}^\infty
c_k\log(1-(-yz)^k)=\Log_c(1-z)\big\vert_{z=-yz}
\end{align*}
The homomor\-phism law for $\Log_c$ is also evident. 
The substitution $(-1)^{k+1}y^k$ for $y$ in the formula for $\Log_c$ is needed to get (\ref{yz}) and will be
useful below. However, to get a logarithm we could have used e.g. $y$
instead. To prove that $\Log_c$ is an isomorphism, assume $f\in P^{\mathbb Q}$. 
Since $S$ is an isomorphism, it is enough to prove that there is a unique $X$ such that $\Log_c(S(X))=f$. 
Let 
\begin{align*}
\varphi_0(z)=\Log_c(\frac{1}{1-z})=c_1z+\text{ higher terms}\\
\varphi_1(z,y)=\Log_c(1+yz)=c_1yz+\text{ higher terms}
\end{align*}
We want to prove that there are uniquely determined $a_k,\ b_k,\ k\geq1$ such that
$$
\sum_{k=1}^\infty a_k\varphi_0(z^k)+\sum_{k=1}^\infty b_k\varphi_1(z^k,y)=f
$$
Suppose inductively that $k\geq1$ and there are uniquely
determined \break $a_1,b_1,\ldots,a_{k-1},b_{k-1}\in \mathbb Q$ such that
$$
\sum_{j=1}^{k-1} a_j\varphi_0(z^j)+\sum_{j=1}^{k-1} a_j\varphi_1(z^j,y)-f=g_0z^k+g_1yz^k + \text{ higher terms}
$$
Since $a_k\varphi_0(z^k)+b_k\varphi_1(z^k,y)=a_kc_1z^k+b_kc_1yz^k+\text{ higher terms}$, 
we may choose $a_k=g_0/c_1$ and $b_k=g_1/c_1$ to get 
$$
\sum_{j=1}^{k} a_j\varphi_0(z^j)+\sum_{j=1}^{k} a_j\varphi_1(z^j,y)\equiv f\mod z^{k+1}
$$

\end{proof}
\sudda{The inverse of
$\Log_c$ is given by 
$$
\Exp_{c'}(X)=\prod_{k=1}^\infty (\exp(X(z^k,(-1)^{k+1}y^k)))^{c'}
$$
where
$$
(\sum_{k=1}^\infty c_k'z^k)\cdot(\sum_{k=1}^\infty c_kz^k)=z
$$}

\begin{prop}
$$
\Lie(X)=\sum_{k=1}^\infty\frac{\mu(k)}{k}\log(X(z^k,(-1)^{k+1}y^k))\quad\text{for all}\quad X\in U.
$$
\end{prop}
\begin{proof}
We will determine $c,$ such that $\Lie=\Log_c.$ In fact, it is
enough to determine $c$ such that $\Lie(X)=\Log_c(X)$ for $X=1-z$ and
$X=1+yz.$ This follows from the fact that $S$ is surjective and both $\Lie$ and $\Log_c$ are
logarithms  which commute with the substitution $z\to z^k.$
Now $\Lie(\frac{1}{1-z})=z$ and $\Lie(1+yz)=yz$ and
\begin{align}\label{log}
-\Log_c(1-z)=-\sum_{k=1}^\infty c_k\log(1-z^k)=\sum_{j,k=1}^\infty
c_k\frac{z^{kj}}{j}=\sum_{n=1}^\infty (\sum_{k|n}kc_k)\frac{z^n}{n}
\end{align}
Hence, $\Lie(1-z)=\Log_c(1-z)$ if $c_1=1$ and $\sum_{k|n}kc_k=0$ for
$n>1.$ M\"obius inversion formula gives that the solution is given by
$kc_k=\mu(k).$ 
With $c$ chosen like this, we may use (\ref{yz}) to get that 
$\Log_c(1+yz)=\Lie(1-z)\big\vert_{z=-yz}=yz=\Lie(1+yz)$ and the result follows.

\end{proof}

From the proof above it follows that any logarithm $L$ is of the form
$\Log_c$ for a unique $c$ if the following conditions hold:
\begin{itemize}
\item $L$ commutes with the substitution $z\to z^k$ for any $k$
\item $L(1+yz)=L(1-z)\big\vert_{z=-yz}$
\end{itemize}

If we apply the formula for $\Lie$ on $1-dz$ we get Serre's formula for
the series of the free Lie algebra on $d$ even generators:
$$
\sum_{n=1}^\infty(\frac{1}{n}\sum_{j|n}\mu(\frac{n}{j})d^j)z^n.
$$ 
We now turn to the series for $\rhc_0(T(V))$. Let $V(z,y)\in P$ be the series for $V$ and let $\hcc(V)$ denote the series for $\rhc_0(T(V))$. Here we use the same name for the series and the corresponding vector space.

By Proposition \ref{fei} we have the following rule for $\hcc$, where we assume that the coefficients in the series are natural numbers 
\begin{align}\label{tsyg}
\hcc(V_1+V_2)=\hcc(V_1)+\hcc(V_2)+\hcc(V_1\cdot V_2/(1-V_1)/(1-V_2)).
\end{align}

\vspace{10pt}
 Now define the transformation $\hcfree:P\to P$ as
$$
\hcfree(V)(z,y)=-\sum_{k=1}^\infty\frac{\phi(k)}{k}\log(1-V(z^k,(-1)^{k+1}y^k)),
$$
where $\phi$ is Euler's $\phi$-function. Observe that this is similar to the formula for $\Lie(1-V)$ above, the difference is only that the M\"obius function has been replaced by 
Euler's $\phi$-function. The transformation $\hcfree$ also satifies  (\ref{tsyg}), which follows from the algebraic identity
$$
1-V_1-V_2=(1-V_1)(1-V_2)(1-V_1\cdot V_2/(1-V_1)/(1-V_2)).
$$
We will now first prove that $\hcc(V)=\hcfree(V)$ when $V$ is one-dimensional. Since both transformations commute with $z\to z^k$, it is enough to consider the two cases $V=z$ and $V=yz$. For a free algebra on one even generator $\rhc_0$ is one-dimensional in each degree $\ge1$, hence $\hcc(z)=z/(1-z)$. Using the computation in (\ref{log}), we see that the series $\hcfree(z)$ is 
$$
\hcfree(z)=\sum_{n=1}^\infty\frac{z^{n}}{n}\sum_{k|n}{\phi(k)}.
$$
Since  $\sum_{k|n}{\phi(k)}=n$ for all $n$, we get that $\hcc(z)=\hcfree(z)$. For a free algebra on one odd generator $\rhc_0$ is zero in each even degree and one-dimensional in each odd degree $\ge1$, hence $\hcc(yz)=yz/(1-z^2)$. Since $(1+yz)(1-yz)=1-z^2$ and $\hcfree(V)$ is logarithmic in $1-V$, we have 
$$
\hcfree(yz)=\hcfree(z^2)-\hcfree(-yz)=\hcfree(z)\big\vert_{z=z^2}-\hcfree(-yz).
$$
Using (\ref{yz}), we get that 
$\hcfree(-yz)=\hcfree(z)\big\vert_{z=-yz}=-yz/(1+yz)$ and hence 
$$
\hcfree(yz)=z^2/(1-z^2)+yz/(1+yz)=yz/(1-z^2)=\hcc(yz).
$$
The proof that $\hcc(V)=\hcfree(V)$ for all $V\in P$ with natural numbers as coefficients is now completed by a double induction, using that both transformations satisfy (\ref{tsyg}). Indeed, fixing  a degree $k$ we claim that the coefficients for $z^k$ and $yz^k$ agree in $\hcc(V)$ and $\hcfree(V)$. For $V=\sum_{r=1}^\infty a_rz^r+b_ryz^r$ define 
$\ord(V)=\min\{r;\ a_r+b_r\ne0\}$ ($\ord(0)=\infty$). We have that 
$\ord(\hcc(V))\ge\ord(V)$ and $\ord(\hcfree(V))\ge\ord(V)$. The claim is true for all $V$ with $\ord(V)>k$. Suppose the claim is true for all $V$ such that 
$\ord(V)>r$ and let $V=a_rz^r+b_ryz^r+W$ where $\ord(W)>r$. We now prove the claim for $V$ by induction over $a_r+b_r$. Suppose the claim is true if $a_r+b_r<N$ and suppose  $a_r+b_r=N\ge1$. Let $V_1=z^r$ if $a_r>0$ and $V_1=yz^r$ otherwise. Then the claim is true for $V_1$ by the above and for $V-V_1$ by induction. Since 
$\ord(V_1(V-V_1))>r$ the claim is also true for $V_1(V-V_1)/(1-V_1)/(1-V+V_1)$ and hence, by (\ref{tsyg}), the claim is true for $V$.

\vskip20pt
\section{Free algebras modulo one relation}
\subsection{Strongly free}

\begin{dfn}
Suppose $A$ is a differential graded $k$-algebra (DGA for short) and $\omega\in A$ is a nonzero homogeneous cycle. A new DGA, $A^\omega$, is defined in the following way. Let $s^\omega A$ denote the graded vector space which is obtained from $A$ by raising all degrees (sign degree, weight and homological degree) with the corresponding degrees for $\omega$. The elements in $s^\omega A$ are written $s^\omega a$ for a homogeneous element $a\in A$ (even $a=1$ is possible) and $s^\omega$ is considered as a linear operator $A\to A^\omega$. Then
\begin{align*}
A^\omega&=k\oplus s^\omega A \quad
\text{as a vector space }\\
d(s^\omega a)&=(-1)^\omega s^\omega d(a)\\
(s^\omega a_1)\cdot(s^\omega a_2)&=s^\omega (a_1\omega a_2).
\end{align*}
\end{dfn}
Now suppose $A$ is just a $k$-algebra without differential. Choose a homogeneous $k$-section $f$ of the natural map $A\to A/A\omega A$ (with $f(1)=1$) and define an algebra homomorphism
\begin{align*}
F:T(s^\omega (A/A\omega A)) \longrightarrow A^\omega \quad\text{induced
by}\quad s^\omega a\mapsto s^\omega f(a).
\end{align*}
We have
$$
F(s^\omega a_1\otimes\cdots\otimes s^\omega a_n)=s^\omega(f(a_1)\omega f(a_2)\omega\ldots\omega f(a_n)).
$$
Since $A$ is generated as an algebra by $\omega$ and the image of $f$ (proof by induction over the weight), we get that $F$ is a surjective homomorphism of algebras. The notion of {\em strongly free} was introduced by D. Anick in \cite{anick}. The slightly different definition given below  is however more suitable for our purposes.  

\begin{dfn}
A homogeneous nonzero element $\omega$ in a $k$-algebra $A$ is called {\em strongly free} if the map $F$ above is an isomorphism. A sequence $\omega=(\omega_1,\omega_2,\ldots)$ of homogeneous elements in $A$ is called a {\em strongly free set} in $A$ if, for all $i$, 
$\omega_i$ is strongly free in $A$ modulo the ideal generated by $\{\omega_1,\ldots,\omega_{i-1}\}$. 
\end{dfn}

\vspace{10pt}
The definition is independent of the choice of the section $f$, since a surjective graded linear map between two isomorphic graded vector spaces is an isomor\-phism.
As always, we use the same notation of a vector space and its series (in two variables). Since $F$ is always surjective, we get the inequality
$$
1/(1-\omega B)\ge A^\omega=1+\omega A
$$
from which the inequality
\begin{align}\label{ineq}
B\ge A/(1+\omega A)
\end{align}
may be deduced. Observe that the inequality of series $C>0$ means that the first nonzero coefficient in $C$ is positive  and that $A\ge B$, $C\ge0$ implies that $AC\ge BC$.

We have that equality holds in (\ref{ineq})
iff $\omega$ is strongly free in $A$. The same formula (and inequality) holds also in the case when $\omega$ is a strongly free set in $A$. Indeed, suppose $\omega=\{\omega_1,\omega_2\}$ is a strongly free set in $A$. Then the series for $A/A\omega A$ is
$$
A/(1+\omega_1 A)/(1+\omega_2A/(1+\omega_1A))=A/(1+\omega_1A+\omega_2A).
$$
It follows from the inequality of series that any reordering of a strongly free set $\omega$ is still a strongly free set and hence the use of the word {\em set} instead of {\em sequence} is appropriate.

In particular, if $\omega$ is any set of homogeneous elements in $T(V)$, then 
$$
T(V)/\ideal{\omega}\ge1/(1-V+\omega)
$$
with equality iff $\omega$ is a strongly free set in $T(V)$.

\vspace{10pt}
\subsection{Free extensions with one variable}
We will now in detail study a free extension with one variable, 
$A\to A\ideal{S;\ dS=\omega}$, 
in particular the case when $A$ is concentrated in homological degree zero and hence $S$ has homological degree one.
The reason for introducing $A^\omega$ in the previous subsection is the following proposition, which also gives another characterization of the property strongly free.

\begin{prop}\label{freeext} Let $i:A\to A\ideal{S;\ dS=w}$ be a free extension of an algebra $A$, concentrated in homological degree zero, with a variable $S$ killing the homogeneous nonzero element $\omega\in A$. Then, for $n,q\ge0$, $\epsilon\in\mathbb Z_2$
\begin{align*}
\ho_n(A\ideal{S})&=\tor_{n+1}^{A^\omega}(k,k)\\
\ho_{n+1,q,\epsilon}((A\ideal{S}/([A\ideal{S},A\ideal{S}]+A),d))&=\rhc_{n,q,\epsilon+1}(A^\omega). 
\end{align*}
\end{prop}
\begin{proof} 
The elements in  $A\ideal{S}$ can uniquely be
represented as  $a_0Sa_1S\ldots Sa_n$, where
$a_0,\ldots,a_n\in A$. The elements in
$(A\ideal{S}/[,]+A,d)$ can uniquely be represented as $a_0Sa_1S\ldots a_nS$, where 
$a_0,\ldots,a_n\in A$. Consider the maps
 $\phi: A\ideal{S}\to\bb(A^\omega)$ of homological degree $+1$ (and weight equal to the weight of $\omega$ and sign degree $|\omega|+1$) 
 and   
$\phi_t:(A\ideal{S}/[,]+A,d)\to\cc_t(A^\omega)$ of homological degree $-1$, weight zero and sign degree one defined by 
\begin{align*}
\phi(a_0Sa_1S\ldots Sa_n)&=(-1)^{n|w|} (s^\omega
a_0|\ldots|s^\omega a_n)\\
\phi_t(a_0Sa_1S\ldots a_nS)&=(-1)^{n|w|} (s^\omega
a_0\otimes\ldots\otimes s^\omega a_n)/(1-t). 
\end{align*}
It is easy to verify that $\phi$ och  $\phi_t$ are  isomorphisms
for complexes.  

\end{proof}

\begin{rmk} The proposition is true also in the case when $A$ is a DGA (the sign for $\phi$ and $\phi_t$ in the proof is changed to $(-1)^{p|w|}$ where $p$ is the homological degree of the argument on the left hand side). 
Then the fact $\ho(A^\omega)=\ho(A)^\omega$
may be used to prove that a set 
$\{\omega_i \}$ is strongly free in A iff $A\ideal{S_i;\ dS_i=\omega_i}$ is acyclic. This could be compared with the result in commutative algebra that a sequence is regular iff the Koszul complex is acyclic.  
\end{rmk}

Since a $k$-algebra $B$ is free iff $\tor^B_n(k,k)=0$ for $n>1$, it follows from Proposition \ref{freeext} that $A\ideal{S;\ dS=w}$ is acyclic iff $\omega$ is strongly free in $A$ and in this case $A\ideal{S;\ dS=w}$ is a relative model for $A\to A/A\omega A$. 
\begin{prop}\label{strong}
Suppose $A$ is a graded algebra and 
$\omega$ is strongly free in   $A$. Put $B=A/A\omega A$ and let $s$ denote shift of sign degree. Then
\begin {align*}
a)&\quad\rhc_n(B)\simeq\rhc_n(A)\quad\text{for}\quad
n\geq2\\
b)&\quad\text{there is an exact sequence}\\
&0\rightarrow s\rhc_1(A)\longrightarrow s\rhc_1(B)
\longrightarrow\rhc_0(A^\omega)\longrightarrow\rhc_0(A)
\longrightarrow\rhc_0(B)\rightarrow0\\
c)&\quad\text{if}\quad\rhc_1(A)=0\quad\text{then}\quad s\rhc_1(B)\cong (\omega A\cap[A,A])/[\omega A,\omega A] 
\end{align*}
\end{prop}
\begin{proof} 
By the argument above and by Proposition \ref{model} and Proposition \ref{freeext}, $\rhc_n(A\to B)\cong s\rhc_{n-1}(A^\omega)$ and then
the long exact sequence for cyclic homology of a map (\ref{longex}) gives
\begin{align*}
&\ldots\to\rhc_n(A)\to\rhc_n(B)\to s\rhc_{n-1}(A^\omega)
\to s\rhc_{n-1}(A)\to\\
&s\rhc_{n-1}(B)\to\rhc_{n-2}(A^\omega)\to \rhc_{n-2}(A)\to\rhc_{n-2}(B)\to\ldots\,.
\end{align*}

Since $A^\omega$ is free, a) and b) follows. 
To prove c), we use the following exact sequence which holds for any homogeneous element $\omega\in A$ and is proved quite elementary.
\begin{align*}
0\rightarrow&(\omega A\cap[A,A])/[\omega A,\omega A]
\rightarrow \omega A/[\omega A,\omega A] \rightarrow\\
&A^+/[A,A]\rightarrow A^+/(\omega A+[A,A]) \rightarrow 0  
\end{align*}
Since $\omega$ is strongly free in $A$, it follows that $\omega$ is a non-zerodivisor and hence the map $s^\omega A\to \omega A$ given by $s^\omega a\mapsto \omega a$ is an isomorphism of vector spaces. Under this map $[A^\omega,A^\omega]$ is mapped to $[\omega A,\omega A]$, since $(s^\omega a)(s^\omega b)=s^\omega(a\omega b)$ is mapped to $\omega a\omega b$. 

Hence, in the sequence above,  
$\omega A/[\omega A,\omega A]$ may be replaced by \break $s^\omega A/[A^\omega,A^\omega]=\rhc_0(A^\omega)$. Comparing the above sequence and the sequence in b) and using  
$\rhc_1(A)=0$ we get 
$$
\rhc_1(B)\cong s((\omega A\cap[A,A])/[\omega A,\omega A]).
$$

\end{proof}
\subsection{Strongly free sets in free algebras}
\begin{prop}\label{freeset}
Suppose $\omega$ is a strongly free set in  $T(V)$ and put $B=T(V)/(\omega)$. Then
\begin{align*}
a)&\quad\rhc_n(B)=0\quad\text{for}\quad n\geq2\\
b)&\quad\rhc_0(B)=y\rhc_1(B)+
\hcfree(V-\omega)\\
c)& \quad\text{if }\omega\text{ consists of monomials, then }\rhc_1(B)=0
\end{align*}
\end{prop}

\begin{proof} The proof is by induction over the size of $\omega$, so suppose the claim is true for $A=T(V)/\ideal{\alpha}$ and that $\omega$ is a strongly free element in $A$ and put $B=A/A\omega A$. Then  a) is true for $B$ by Proposition \ref{strong}.  To prove b) we use the same proposition to get
$$
\rhc_0(B)-y\rhc_1(B)=\rhc_0(A)-y\rhc_1(A)-\rhc_0(A^\omega).
$$
Since $A^\omega$ is free on $s^\omega B$ and $B$ has series $1/(1-V+\alpha+\omega)$ by the above, we get by induction
$$
\rhc_0(B)-y\rhc_1(B)=\hcfree(V-\alpha)-\hcfree(\omega/(1-V+\alpha+\omega))
$$ 
We now use that $\hcfree(V)$ is logarithmic in $1-V$ and the fact that
$$
(1-V+\alpha)/(1-\omega/(1-V+\alpha+\omega))=1-V+\alpha+\omega
$$
to get 
$$
\rhc_0(B)-y\rhc_1(B)=\hcfree(V-\alpha-\omega)
$$
which proofs the induction step for b). The induction start, $\omega=0$, for a), b) and c) is true, since $\rhc_n(T(V))=0$ for $n>0$ and since 
$\rhc_0(T(V))=\hcfree(V)$ by definition of hcfree.

To prove the induction step for c) we will use Proposition \ref{strong} and prove that $\omega A\cap[A,A]\subseteq[\omega A,\omega A]$. 

Since $A$ is a free algebra modulo monomials it has a fine grading where the graded components consist of the nonzero monomials in $A$. Consider the coarser grading where the components (called {\em cyclic components}) consist of the direct sum of the cyclic permutations of a given monomial. Then $[A,A]$ respects this grading and hence $a\in[A,A]$ iff each cyclic component of $a$ is in $[A,A]$. 

Suppose $x$  is a nonzero element in $\omega A\cap[A,A]$ lying in a cyclic component $C$ and let 
$\omega a$ be a nonzero monomial  term in $x$. We claim that all cyclic permutations of $\omega a$ are nonzero, which implies that $C$ has dimension equal to the length of $\omega a$ and hence the cyclic 
operator $t$ circulates a basis for $C$. Here $t$ is defined on a monomial $x_1x_2\ldots x_n$ as
$$
t(x_1x_2\ldots x_n)=(-1)^{|x_1|(|x_2|+\cdots |x_n|)}x_2x_3\ldots x_nx_1
$$
Since $C\cap[A,A]=\im(1-t)$, the claim also implies that 
$$
x=\sum_{i=0}^{n-1}c_it^i(\omega a)\quad\text{where}\quad\sum_{i=0}^{n-1}c_i=0.
$$ 

Indeed, if the claim is not true, then there is an overlap of $\omega$  with a monomial in the set $\alpha$ which defines $A$ or $\omega$ is a submonomial of some monomial in $\alpha$ or $\omega$ overlaps itself. In each case it is easy to construct a nonzero element in $\ker(F)$; e.g., if $\alpha_1\alpha_2\in\alpha$, 
$\omega=\alpha_2\omega_2$ and $\alpha_1$ is nonzero as an element in $A/A\omega A$ then $F(s^\omega(\bar\alpha_1)s^\omega(1))=0$. 

Moreover, since $\omega$ is strongly free in $A$, the occurences of $\omega$ in $\omega a$ do not overlap, so we may write 
$$
\omega a=\omega a_1\omega a_2\ldots\omega a_r
$$
where for all $i$, $a_i$ does not contain $\omega$ as a submonomial. Since each monomial term in $x$ belongs to $\omega A$ we have ($n_0=0$):
\begin{align}\label{xx}
x=\sum_{i=0}^{r-1}d_{i}t^{n_i}(\omega a)\ \text{ where }\ \sum_{i=0}^{r-1}d_{i}=0\ \text{ and }\  n_i=\len(\omega a_1\omega a_2\ldots\omega a_i)
\end{align}
If $bc$ is a monomial and $\len(b)=m$, then it is easy to prove by induction over $m$ that 
$$
t^m(bc)=(-1)^{|b||c|}cb
$$
Hence, for $i=0$ to $r-1$, we have 
$$
t^{n_i}(\omega a)=(-1)^{|b_i||c_i|}c_ib_i
$$
where $b_i=\omega a_1\omega a_2\ldots\omega a_i$ and $c_i=\omega a_{i+1}a_{i+2}\ldots\omega a_r$.

Since one may compute $|b_i||c_i|+|b_{i+1}||c_{i+1}|$ to be 
$|\omega a_{i+1}|(|b_i|+|c_{i+1}|)$ it follows that
\begin{align*}
t^{n_i}(\omega a)-t^{n_{i+1}}(\omega a)&=(-1)^{|b_i||c_i|}(\omega a_{i+1}c_{i+1}b_i-
(-1)^{|\omega a_{i+1}|(|b_i|+|c_{i+1}|)}c_{i+1}b_i\omega a_{i+1})\\
&=(-1)^{|b_i||c_{i}|}[\omega a_{i+1},c_{i+1}b_i]\in[\omega A,\omega A]
\end{align*}

By (\ref{xx}) we have
$$ 
x=\sum_{i=0}^{r-1}(\sum_{j=0}^ic_j)(t^{n_i}(\omega a)-t^{n_{i+1}}(\omega a))
$$
and it follows that $x\in[\omega A,\omega A]$.

\end{proof} 

Observe that D. Anick has proved that a strongly free set of monomials is just a set of monomials such that no two monomials have a common beginning and ending and no monomial is contained in another monomial. Moreover, if there is an order of the variables such that the leading monomials of a set $\omega$ of homogeneous polynomials yield a strongly free set of monomials, then $\omega$ is strongly free.

\section{Generic quadratic forms}
As an application of the results in the previous section, we will compute the cyclic homology of a free algebra modulo some generic quadratic forms. We call some forms generic if the $N$ coefficients are algebraically independent, or that the coefficients defines a point in $k^N$ which lies in a non-empty Zariski-open subset of $K^N$. We will assume that the variables are odd (in which case the Koszul dual has even variables). 

Assume 
$\omega$ is a generic quadratic form in $T(V)=k\ideal{T_1,\ldots T_n}$ with $n\ge2$. Then $\omega$ is strongly free in $T(V)$, since the Hilbert series is minimal for a generic relation and there is an example of a strongly free form, namely  
$\omega=T_1T_2$. By Proposition \ref{freeset} we get
$$
\rhc_0(T(V)/(\omega))\ge\hcfree(nzy-z^2)
$$
But since $\rhc_0(T(V)/(\omega))=T(V)/[T(V),T(V)]+(\omega)$ it follows that the series for $\rhc_0(T(V)/(\omega))$ is the minimal possible in the generic case. Since by Proposition \ref{freeset}, $\rhc_0(T(V)/(T_1T_2)=\hcfree(nzy-z^2)$, we get 
\begin{prop} If $\omega$ is a generic quadratic form in $T(V)=k\ideal{T_1,\ldots T_n}$ with $n\ge2$, then
\begin{align*}\rhc_0(T(V)/(\omega))&=\hcfree(nzy-z^2)\\
  \rhc_n(T(V)/(\omega))&=0\quad\text{for}\quad n>0
\end{align*}
\end{prop} \endproof
See also Proposition \ref{genstr}, which is a generalization.

\vspace{10pt}
In the symmetric case we have the following
\begin{prop} Suppose $\omega$ is a generic symmetric quadratic form in $n$ odd variables of weight one, $a_1,\ldots,a_n$, where $n\ge3$, and let $A=k\ideal{a_1,\ldots,a_n}$ and $B=A/(\omega)$. Then
\begin{align*}\rhc_0(B)&=z^2+\hcfree(nzy-z^2)\\
 \rhc_1(B)&=yz^2\\
  \rhc_n(B)&=0\quad\text{for}\quad n>1
\end{align*}
\end{prop} 
\begin{rmk} In fact, it is enough to assume that the symmetric form has rank at least three in the proposition above.
\end{rmk}
\begin{proof}
 As above it follows that $\omega$ is strongly free since $[a_1,a_2]$ is strongly free. Since $\omega\in [A,A]$, it follows that $\omega$ is a non-zero element in 
 $\omega A\cap[A,A]/[\omega A,\omega A]$ and hence by Proposition \ref{strong} $\rhc_1(B)\ge yz^2$ and by Proposition \ref{freeset}, 
$$
\rhc_0(B)\ge z^2+\hcfree(nzy-z^2)
$$
Hence, it is enough to give an example of a $B$ which attains the lower bound. We will consider $B=k\ideal{a,b,c}/([a,b]+c^2)$. We have that $[a,b]+c^2$ is strongly free in $k\ideal{a,b,c}$ since $ab$ is a strongly free leading monomial. It is enough to prove that 
$\rhc_0(B)=z^2+\hcfree(3zy-z^2)$, since then by Proposition \ref{fei},
\begin{align*}
\rhc_0(B*k\ideal{a_4,\ldots a_n})=&z^2+\hcfree(3zy-z^2)+\hcfree((n-3)zy)+\\
&\hcfree(B^+(n-3)zy/(1-(n-3)zy))
\end{align*}
But $B=1/(1-3zy+z^2)$ and 
\begin{align*}
&1-(3zy-z^2)(n-3)zy/(1-(n-3)zy)/(1-3zy+z^2)=\\
&\quad(1-nzy+z^2)/(1-3zy+z^2)/(1-(n-3)zy)
\end{align*} 
and then, since $\hcfree(V)$ is logarithmic in $1-V$,
$$
\rhc_0(B*k\ideal{a_4,\ldots a_n})=z^2+\hcfree(nzy-z^2)
$$
A basis for $\rhc_0(k\ideal{a,b}/ab)$ is easily seen to be $\{b^j,a^k;\ j \text{ and }k \text{ odd }\ge1\}$. By Proposition \ref{fei} a basis for 
$\rhc_0(k\ideal{a,b,c}/ab)$ is 
$$
\{b^j,a^k, c^l;\ j,k,l\text{ odd }\ge1\}\cup\{(b^{j_1}a^{k_1}c^{l_1})\ldots
(b^{j_r}a^{k_r}c^{l_r});\ j_i+k_i, l_i>0\}
$$
where the last set is taken modulo cyclic permutation of groups. The series for the basis is $\hcfree(3zy-z^2)$. Considering the relation $[a,b]+c^2$ as the re-write rule $ab\to -ba-c^2$ one can see that the elements in the set above together with 
$\{b^ja^k\}$ generate $\rhc_0(B)$. We will now prove that for any $j>0,k>0,j+k>2$ we either do not need the generator $b^ja^k$ or  not $b^{j-1}a^{k-1}c^2$. From this the result follows since then the series for $\rhc_0(B)$ is at most $z^2+\hcfree(3zy-z^2)$ which is also  the minimal possible value.

The following computation is done modulo $c^2$. We have,
\begin{align*}
b^ja^k=(-1)^{j+k-1}ab^ja^{k-1}=(-1)^{k-1}b^ja^k\\
b^ja^k=(-1)^{j+k-1}b^{j-1}a^{k}b=(-1)^{j-1}b^ja^k
\end{align*}
Hence, if $j$ or $k$ is even, then $b^ja^k$ may be expressed in terms of the generators containing $c^2$. Suppose now $j$ and $k$ are odd.  We have,
\begin{align*}
b^ja^k=-ab^ja^{k-1}=b^ja^k+\sum_{i=1}^j(-1)^ib^{i-1}c^2b^{j-i}a^{k-1}
\end{align*}
and
$$
b^{i-1}c^2b^{j-i}a^{k-1}=(-1)^{(i-1)(j-i+k-1)}b^{j-i}a^{k-1}b^{i-1}c^2
$$
But modulo $c^4$ we have 
$$
b^{j-i}a^{k-1}b^{i-1}c^2=(-1)^{(i-1)(k-1)}b^{j-1}a^{k-1}c^2
$$
whence
$$
(-1)^ib^{i-1}c^2b^{j-i}a^{k-1}=(-1)^{i+(i-1)(i-1)}b^{j-1}a^{k-1}c^2=-b^{j-1}a^{k-1}c^2
$$
and hence $-jb^{j-1}a^{k-1}c^2$ belongs to the ideal generated by $c^4$ in $k\ideal{a,b,c}$ which implies that $b^{j-1}a^{k-1}c^2$ is not needed as a generator ($\ch(k)=0$).

\end{proof}
\begin{rmk} The anti-symmetric case, i.e., when the variables are even, may be treated similarly (and more easily) by studying $\omega=[a,b]+[c,d]$ in $k\ideal{a,b,c,d}$. There might however be some more exceptional cases in three variables. 
\end{rmk}
The results are easy to generalize to cover cases with several (but not too many) generic relations. We have,
\begin{prop}\label{genstr} Suppose there is a strongly free set of monomials in $T(V)$ with series $\omega(z,y)$. Then for any set $g$ of generic polynomials with series $\omega(z,y)$ we have,
\begin{align*}\rhc_0(T(V)/(g))&=\hcfree(V-\omega)\\
  \rhc_n(T(V)/(g))&=0\quad\text{for}\quad n>0
\end{align*}
\end{prop}
\begin{proof}
Put $B=T(V)/(g)$. We have $B\ge1/(1-V+\omega)$. By assumption, the minimal value is attained and hence $B=1/(1-V+\omega)$ since $g$ is generic. It follows that  $g$ is a strongly free set. Also $\rhc_0(B)$ is minimal in the generic case and by Proposition \ref{freeset}, $\rhc_0(B)\ge\hcfree(V-\omega)$. Again, this minimal value is attained by assumption  and Proposition \ref{freeset}, hence it is also the series in the general case. Moreover, Proposition \ref{freeset} also gives that $\rhc_n(B)=0$ for $n>0$.
\end{proof}
\begin{prop}
Suppose $\omega_1,\ldots,\omega_r$ are $r$ generic symmetric quadratic forms in $n$ odd variables of weight one, where $r\le jk$ and $j+k+jk\le n$ for some positive integers $j,k$ and let $A=k\ideal{a_1,\ldots,a_n}$ and 
$B=A/(\omega_1\ldots,\omega_r)$. Then 
\begin{align*}\rhc_0(B)&=rz^2+\hcfree(nzy-rz^2)\\
 \rhc_1(B)&=ryz^2\\
  \rhc_n(B)&=0\quad\text{for}\quad n>1
\end{align*}
\end{prop}
\begin{proof}
Consider $\omega_{m,n}=[a_m,b_n]+c_{m,n}$ for $m=1,\ldots,j$ and $n=1,\ldots,k$ in 
$k\ideal{a_1,\ldots,a_j,b_1,\ldots,b_k,c_{1,1},\ldots,c_{j,k}}$. 
 
\end{proof}
Taking Koszul dual, we obtain a result for quotients of a polynomial ring by sufficiently many generic quadratic relations.
\begin{prop} Suppose $f_1,\ldots,f_t$ are generic quadratic relations in $A=k[x_1,\ldots,x_n]$, where $x_i$ are even variables of weight one and $t={n+1\choose2}-r$ and where $r\le jk$, $j+k+jk\le n$ for some positive integers $j,k$. Let $R=A/(f_1,\ldots,f_t)$ and $\hcfree(nzy-rz^2)=\sum_{i\ge1}a_iz^i+y\sum_{i\ge1}b_iz^i$. Then
\begin{align*}\rhc(R)&=rz^2+rxyz^2+\sum_{i\ge1}b_ix^{i-1}z^i+y\sum_{i\ge1}a_ix^{i-1}z^i
\end{align*}
\end{prop}
\noindent\textbf{Example}
Let $R=k[x_1,\ldots,x_7]/(f_1,\ldots,f_{25})$ where $f_1,\ldots,f_{25}$ are generic quadratic polynomials in the even variables $x_1,\ldots,x_7$. Then
$$
\hcfree(7yz-3z^2)=18 z^2 + 465 z^4 + \ldots +y (7 z + 98 z^3 + 2401 z^5+\ldots)
$$
and
$$
\rhc(R)=7z+3z^2+21yxz^2+98x^2z^3+465yx^3z^4+2401x^4z^5+\ldots
$$

\vspace{20pt}
\noindent \textbf{Exceptional cases}: 
\begin{align*}
&A_0=k\ideal{T_1,\ldots,T_n}/(T_1^2)\quad T_i\quad\text{even }\\
&A_1=k\ideal{T_1,\ldots,T_n}/(T_1^2)\quad T_i\quad\text{odd }\\
&B_0=k\ideal{T_1,\ldots,T_n}/([T_1,T_2])\quad T_i\quad\text{even }\\
&B_1=k\ideal{T_1,\ldots,T_n}/([T_1,T_2])\quad T_i\quad\text{odd }\\
\\
&\rhc(A_0)(x,y,z)=z/(1-x^2z^2)+\hcfree((n-1)z+(n-1)z^2) \\
&\rhc(A_1)(x,y,z)=yz/(1-xz)+\hcfree((n-1)yz+(n-1)z^2) \\
&\rhc(B_0)(x,y,z)=z^2/(1-z)^2+yxz^2/(1-z)^2+\hcfree(nz-z^2) \\
&\rhc(B_1)(x,y,z)=z^2/(1-z^2)^2+2yz/(1-z^2)+yxz^2/(1-z^2)^2+\\
&\hspace{3.2cm}+\hcfree(nyz-z^2) 
\end{align*}
\begin{proof}
For each case we start to assume $n=1$ for $A_0,A_1$ and $n=2$ for $B_0,B_1$. Except for $B_0$ we first compute the series for the Koszul dual. We have that $A_0^!$ is free on one odd generator hence, $\rhc(A_0^!)=yz/(1-z^2)$ and by (\ref{hckos}), 
$\rhc(A_0)=z/(1-x^2z^2)$. 

Now, by Proposition \ref{fei} we get for general $n$,
\begin{align*}
&\rhc(k\ideal{T_1}/(T_1^2)*k\ideal{T_2,\ldots,T_n})=z/(1-x^2z^2)+\hcfree((n-1)z)+\\  &\hcfree(z(n-1)z/(1-(n-1)z))
\end{align*}
but 
$$
(1-(n-1)z)(1-(n-1)z^2/(1-(n-1)z))=1-(n-1)z-(n-1)z^2
$$
whence the result, since $\hcfree(V)$ is logarithmic in $1-V$.

The proof of $A_1$ is similar. The algebra $B_0$ for $n=2$ is a commutative polynomial ring in two even variables, hence by (\ref{hkr}), $\hho(B_0)=(1+xyz)^2/(1-z)^2$. Now by (\ref{cyc}), 
$$
\rhc(B_0)=((1+xyz)^2/(1-z)^2-1)/(1+xy)=(2z-z^2+xyz^2)/(1-z)^2.
$$
Observe that $1+xy=y(x+y)$ so in order to divide a polynomial $p(x)$ by $1+xy$ one may instead perform the division of $p(x)$ by $x+y$ (and check divisibility by putting $x=-y$).

We have for a general $n$,
\begin{align*}
&\rhc(k\ideal{T_1,T_2}/([T_1,T_2])*k\ideal{T_3,\ldots,T_n})=
(2z-z^2+xyz^2)/(1-z)^2+\\
&\hcfree((n-2)z)+  
\hcfree((1/(1-z)^2-1)(n-2)z/(1-(n-2)z))
\end{align*} 
but
$$
(1-(n-2)z)(1-(2z-z^2)(n-2)z/(1-z)^2/(1-(n-2)z))=(1-nz+z^2)/(1-z)^2
$$
whence, 
\begin{align*}
&\hcfree((n-2)z)+  
\hcfree((1/(1-z)^2-1)(n-2)z/(1-(n-2)z))=\\
&\hcfree(nz-z^2)-2\hcfree(z)
\end{align*}
and the result for $B_0$ follows since, $\hcfree(z)=z/(1-z)$.

For $B_1$ and $n=2$ the Koszul dual is $B_1^!=k[a]/(a^2)\otimes k[b]/(b^2)$ where $a,b$ are even. Hence by (\ref{tensor}), $\hho(B_1^!)=(\hho(k[a]/(a^2)))^2$ and by the result for $A_0$,
$$
\hho(k[a]/(a^2))=1+(1+xy)\rhc(k[a]/(a^2))=1+(1+xy)z/(1-x^2z^2)
$$
Hence,
$$
\hho(B_1^!)=1+(1+xy)^2z^2/(1-x^2z^2)^2+2(1+xy)z/(1-x^2z^2)
$$
and then by (\ref{cyc})
$$
\rhc(B_1^!)=(1+xy)z^2/(1-x^2z^2)^2+2z/(1-x^2z^2)
$$
and finally by (\ref{hckos})
\begin{align*}
\rhc(B_1)&=y/x((1+y/x)x^2z^2/(1-z^2)^2+2xz/(1-z^2))=\\
&=yxz^2/(1-z^2)^2+z^2/(1-z^2)^2+2yz/(1-z^2)
\end{align*}
A similar computation as for $B_0$ now gives the result for $B_1$ for general $n$.
\end{proof}

\vspace{20pt}\noindent
Clas L\"ofwall\\
Gim\aa{}gatan 5\\
12848 Bagarmossen\\
Sweden

\vspace{10pt}\noindent
e-mail:\\
clas.lofwall@gmail.com

\end{document}